\documentclass[runningheads]{llncs}
\usepackage[utf8]{inputenc}
\usepackage[T1]{fontenc}
\usepackage{amsmath}
\usepackage{amsfonts}
\usepackage{amssymb}
\usepackage{makeidx}
\usepackage{graphicx}
\usepackage[T1]{fontenc}
\usepackage{csquotes}
\usepackage{hyperref}
\usepackage[english]{babel}
\usepackage{bbold}
\usepackage{cases}
\usepackage{enumitem}
\usepackage{multicol}

\usepackage{graphicx}

\usepackage{amsmath,amscd,amsxtra,latexsym,mathrsfs}
\usepackage{dsfont}
\usepackage{mathabx}
\usepackage{stmaryrd}

\newcommand{\esp}{\mathbb{E}}

\newcommand{\ad}{\mbox{ad}}

\newcommand{\Xt}{X^{\tau}}
\newcommand{\nut}{\nu^{\tau}}
\newcommand{\sigt}{\sigma^{\tau}}

\newcommand{\Wt}{W^{\tau}}

\newcommand{\EE}{{\mathbb E}}
\newcommand{\E}{{\mathbb E}}

\newcommand{\PP}{{\mathbb P}}

\newcommand{\SE}{{\mathscr E}}
\newcommand{\SG}{{\mathscr G}}

\newcommand{\SH}{{\mathscr H}}

\def\di{\displaystyle}
\def\f{\frac}

\def\n{\nabla }

\def\s{\sigma }

\def\mathpal#1{\mathop{\mathchoice{\text{\rm #1}}%
   {\text{\rm #1}}{\text{\rm #1}}%
   {\text{\rm #1}}}\nolimits}

\def\Id{\mathpal{Id}}

\usepackage{xcolor}

\DeclareMathOperator{\Hess}{Hess}
\usepackage{graphicx}
\begin{document}
\title{Continuous-time filtering in Lie groups : estimation via the Fréchet mean of solutions to stochastic differential equations}
\titlerunning{Continuous time filtering in Lie groups : estimation via Fréchet mean}
%
\author{Marc Arnaudon\inst{1} \and Magalie Bénéfice\inst{2} \and
Audrey Giremus\inst{3}}
\authorrunning{M. Arnaudon, M. Bénéfice and A. Giremus}
%
\institute{Univ. Bordeaux, CNRS, Bordeaux INP, IMB, UMR 5251, F-33400 Talence, France \and Université de Lorraine, CNRS, IECL, F-54000 Nancy, France \and Univ. Bordeaux, CNRS, Bordeaux INP, IMS, UMR 5218, F-33400 Talence, France}
\maketitle              
\begin{abstract}
We compute the Fréchet mean $\SE_t$ of the solution $X_t$ to a continuous-time stochastic differential equation in a Lie group. It provides an estimator with minimal variance of $X_t$. We use it in the context of Kalman filtering and more precisely to infer rotation matrices. In this paper, we focus on the prediction step between two consecutive observations. Compared to state-of-the-art approaches, our assumptions on the model are minimal.
\keywords{ Filtering \and Lie groups  \and rotation matrices \and Fr\'echet mean \and Prediction }
\end{abstract}

\section{Introduction}
\label{S0}
Optimal filtering has been an active area of research since the introduction of the Kalman filter in the 1960s \cite{K:60}, with applications spanning from navigation and target tracking to industrial control. Over the past decades, different variants have been proposed to overcome the limitations of the seminal algorithm. They make it possible to deal with non-Gaussian uncertainties or non-linearities in the observation or dynamic models, which jointly define the state space representation. More recently, the topic of inferring parameters that do not lie on Euclidean spaces has emerged, with a focus on Lie groups. This issue occurs for instance when estimating rotation matrices, as required in inertial navigation, or, more generally, rigid body transformations such as the pose of a camera in computer vision. Filters designed for that purpose either provide analytical solutions at the cost of restrictive assumptions on the state space representation or leverage the Bayesian formalism to calculate approximate estimates and confidence intervals. In the first case, either invariant or equivariant properties are usually required \cite{Ba:17,VG:23}, whereas in the latter case, the posterior distribution of the parameters to be inferred, conditionally upon the measurements, is enforced to be a concentrated Gaussian distribution in several works \cite{Bo:15}. Conversely, this paper sets the basis for an alternative continuous-time filter on Lie groups that relaxes the assumptions on the state space representation while yielding estimates accurate up to the second order. We focus on the calculation of the prediction of the parameters of interest in the time interval between two consecutive observations. More precisely, we establish a system of differential equations that describes the joint evolution of their Fr\'echet mean, taken as their estimate, and of the covariance matrix of the estimation error in the Lie algebra.

The outline of the paper is the following: Section \ref{S1} is dedicated to prerequisites on the Fr\'echet mean on Lie groups; Section \ref{S2} addresses its calculation in the case of a generic stochastic differential equation on Lie groups; proof is provided in Section \ref{S3} and Section \ref{Ssim} presents simulation results while conclusions and perspectives are proposed in Section \ref{Sfinal}.

\section{Prerequisites:  Fr\'echet means on Lie groups}\label{S1}

Let $G$ be a Lie group with Lie algebra $\SG$ and neutral element~$e$.
\begin{definition}
\label{D1}
For a random variable $X$ on $G$ with support ${\rm supp }(X)$, we say that $\SE(X)\in G$ is the exponential barycenter of $X$ if 
there exists a  $\SG$-valued integrable and centered random variable $\nu(X)$  such that 
\begin{equation}
    \label{E1}
    X=\SE(X)\exp(\nu(X)),
\end{equation}
the segment $\{\SE(X)\exp(t\nu(X)),\  t\in [0,1]\} $ is included in the convex hull of ${\rm supp }(X)$
and the couple $(\SE(X),\nu(X))$ is unique.
\end{definition}
It is well known (\cite{EM:91} Propositions  4 and 5) that if the support of $X$ is sufficiently small, then the exponential barycenter of $X$ exists (and is unique by definition).

In the sequel we will assume that $G$ has a left invariant metric, and we will denote by $\langle\cdot,\cdot\rangle_g$ the scalar product at point $g\in G$. 
\begin{definition}
\label{D2}
An open ball $B(g,R)$ of $G$ centered at $g$ with radius $R>0$ is said to be a regular geodesic ball if $\di R<\f{\pi}{2\sqrt{K}}$ where $K>0$ is an upper bound of the sectional curvatures in $G$, and $\forall x\in B(g,R)$, the cutlocus of $x$ does not meet $B(g,R)$.  
\end{definition}

In this situation, we have the following result (\cite{K:90} Theorem 4.2 and \cite{P:94} Proposition 3.6)
\begin{proposition}
\label{P1}
If the support of the random variable $X$ is included in a regular geodesic ball of $G$, then $\SE(X)$ is the unique minimizer in $\SH(X)$ of $\di g\mapsto \EE\left[ \rho^2(g,X)\right]$, where $\rho$ is the Riemannian distance in $G$. Moreover, we have 
\begin{equation}
    \label{E2}
  \EE\left[ \rho^2(\SE(X),X)\right]=\EE[\|\nu\|^2].
\end{equation}
In other words, $\SE(X)$ is the Fréchet mean of $X$.
Consequently, $\SE(X)$ is the estimator of $X$ minimizing the variance of the error.
\end{proposition}

\section{Computing the Fréchet mean of a stochastic differential equation}\label{S2}

Consider the Stratonovich stochastic differential equation (SDE) in~$G$
\begin{equation}
    \label{E3}
    \circ dX_t= X_t\left(b(X_t)\, dt +\circ dW_t\right), \qquad X_0=e
\end{equation}
with $b\in C^\infty(G,\SG)$, $\di W_t=\sum_{i=1}^m \s_i W_t^i$ where for all $i$, $\s_i\in \SG$, and the $(W_t^i)_{1\le i\le m}$ are independent real-valued Brownian motions. Denote by $\tau$ the exit time of a fixed regular geodesic ball $B(e,R)$.   Denote by $\SE_t=\SE(X_t^\tau)$ the exponential barycenter of the stopped process $X_t^\tau=X_{t\wedge\tau}$ and $\nu_t^\tau=\nu(X_t^\tau)$ the error of the estimator. 
For $y,x\in B(e,R)$, denote $\log(y,x)=y^{-1}\log_y(x)\in \SG$. Notice that 
\begin{equation}
    \label{E4}
    \nu_t^\tau=\log(\SE_t,X_t^\tau).
\end{equation}
We will also denote , $b^\tau(X_t)=b(X_t)1_{\{t<\tau\}}$ and all  $\s_i^\tau=\s_i1_{\{t<\tau\}}$, $i=1,\ldots, m$.
The following theorem yields an ordinary differential equation for the Fréchet mean, which is the basis for computations.
\begin{theorem}
\label{T1}
The Fréchet mean $\SE_t$ of the solution to Equation~\eqref{E3} is solution to the McKean-Vlasov equation $\di d\SE_t=\SE_t h_t\, dt$ with 
 \begin{equation}\label{E5}
        \begin{split}
    &h_t=-\EE\left[T_1\log(\SE_t\cdot, 0_{X_t^\tau})\right]^{-1}\\
    &\cdot \EE\left[T_2\log(0_{\SE_t},\cdot_{X_t^\tau})(X_t^\tau b^\tau (X_t^\tau))+\f12\sum_{i=1}^m\Hess_2\log(0_{\SE_t},\cdot_{X_t^\tau})(X_t^\tau\s_i^\tau\otimes X_t^\tau\s_i^\tau) \right].
\end{split}
    \end{equation}
    Here $T_1$, $T_2$ and $\Hess_2$ denote respectively derivative with respect to first and second variable, and Hessian with respect to the second variable.
\end{theorem}
\begin{proof}
The existence, uniqueness and smoothness of $\SE_t$ with respect to the law of $X_t^\tau$ is garanteed by the fact that ${\rm supp}(X_t^\tau)$ is included in a regular geodesic ball (\cite{K:90} Theorem 4.2, \cite{EM:91} Propositions 4 and 5). On the other hand we have 
\begin{equation}
    \label{E10}
    \circ d\nu_t^\tau=T\log(d\SE_t,\circ dX_t^\tau)
\end{equation}
yielding in It\^o form
\begin{equation}
    \label{E11}
    \begin{split}
     d\nu_t^\tau=&T_1\log(\cdot,0_{X_t^\tau})(d\SE_t)+T_2\log (0_{\SE_t},\cdot)(X_tb_t^\tau(X_t))dt\\&+\f12\sum_{i=1}^m\Hess_2\log(0_{\SE_t},\cdot_{X_t^\tau})(X_t^\tau\s_i^\tau\otimes X_t^\tau\s_i^\tau)dt+dM_t
     \end{split}
\end{equation}
with $M_t$ a martingale. Since $\di \EE[\nu_t^\tau]\equiv 0$ we get from~\eqref{E11}
\begin{equation}
    \label{E12}
    \begin{split}
     0=&\EE\left[T_1\log(\cdot,0_{X_t^\tau})\right](\SE_t h_t)+\EE\left[T_2\log (0_{\SE_t},\cdot)(X_t^\tau b_t^\tau(X_t))\right]\\&+\f12\EE\left[\sum_{i=1}^m\Hess_2\log(0_{\SE_t},\cdot_{X_t^\tau})(X_t^\tau\s_i^\tau\otimes X_t^\tau\s_i^\tau)\right].
     \end{split}
\end{equation}
By \cite{AL:05} Lemma 2.6, the endomorphism of $\SG$:   $\EE\left[T_1\log(\SE_t\cdot,0_{X_t^\tau})\right]$ is invertible. From this last point we get~\eqref{E5}, and this achieves the proof of Theorem~\ref{T1}.
\qed
\end{proof}
With Theorem~\ref{T1} at hand, we propose to look for computable approximations of the Fréchet means $\SE_t$ of the solution~$X_t$ to the stochastic differential equation~\eqref{E3}, together with the error term~$\nu_t$.

\begin{theorem}
\label{T2}
We have the expansions $d\SE_t=\SE_th_t\, dt$ with
    \begin{equation}\label{E6}
        \begin{split}
          &h_t=  b(\SE_t)\\&+\f12\left(\Hess_{\SE_t}b\left(\SE_t(\cdot),\SE_t(\cdot)\right)-\left[ T_{\SE_t}b(\SE_t(\cdot),\cdot\right]+\f1{8}\sum_{i=1}^m[[[\s_i,\cdot],\cdot ],\s_i]
         \right)\cdot\EE[\nu_t^\tau\otimes \nu_t^\tau]\\&+O(t^2).
        \end{split}
    \end{equation}
    For the error term we have
    \begin{equation}\label{E7}
        \begin{split}
       & \f{d}{dt} \EE[\nu_t^\tau\otimes \nu_t^\tau]\\
       &=
       \left\{
       \left(2T_{\SE_t}b(\SE_t (\cdot))-2[b(\SE_t),\cdot]-\f13\sum_{i=1}^m[[\s_i,\cdot],\s_i]\right)\odot \Id
       \right\}
       \EE[\nu_t^\tau\otimes \nu_t^\tau]\\
       &+\f14\left(\sum_{i=1}^m[\s_i,\cdot]\otimes [\s_i,\cdot]\right)\EE[\nu_t^\tau\otimes \nu_t^\tau]+\sum_{i=1}^m\s_i\otimes \s_i+ O(t^2)
        \end{split}
    \end{equation}
    where $A\odot B=\frac12(A\otimes B+B\otimes A)$ is the symmetric tensor product.
\end{theorem}
The proof of Theorem~\ref{T2} is postponed to Section~\ref{S3}

The last Theorem~\ref{T3} specializes Theorem~\ref{T2} to the case $G=SO(3)$. The Lie group involved in inertial navigation, which is our main application, is indeed isometric to products of two copies of $SO(3)$ together with Euclidean spaces. 
Note that the simulations in Section \ref{Ssim} will be done on $SO(3)$. 

\begin{theorem}
\label{T3}
Assume that $G=SO(3)$ endowed with its canonical bi-invariant metric. Also assume that $(\s_1,\s_2,\s_3)=\s (G_1,G_2,G_3)$ where $(G_1,G_2,G_3)$ is an orthonormal basis of $\SG$ and $\s>0$. Then the expansions of Theorem~\ref{T2} take the form $d\SE_t=h_t\, dt$ with
    \begin{equation}\label{E8}
        \begin{split}
          &h_t=  b(\SE_t)+\f12\left(\Hess_{\SE_t}b\left(\SE_t(\cdot),\SE_t(\cdot)\right)-\left[ T_{\SE_t}b(\SE_t(\cdot),\cdot\right]\right)\cdot\EE[\nu_t^\tau\otimes \nu_t^\tau]+O(t^2)
        \end{split}
    \end{equation}
    and for the error term,
    \begin{equation}\label{E9}
        \begin{split}
       & \f{d}{dt} \EE[\nu_t^\tau\otimes \nu_t^\tau]=\s^2\Id+
       \left\{
       2\left(T_{\SE_t}b(\SE_t (\cdot))-[b(\SE_t),\cdot]\right)\odot \Id
       \right\}
       \EE[\nu_t^\tau\otimes \nu_t^\tau]\\
       &+\f18\EE\left[\|\nu_t^\tau\|^2\Id_{(\nu_t^\tau)^\perp}\right]+ O(t^2).
        \end{split}
    \end{equation}
In terms of the Levi-Civita connection $\n$ and the curvature tensor $R$, Equations~\eqref{E8} and~\eqref{E9} rewrite
\begin{equation}\label{E13}
        \begin{split}
          &h_t=  b(\SE_t)+\left(\f12\Hess_{\SE_t}b\left(\SE_t(\cdot),\SE_t(\cdot)\right)+\n_\cdot T_{\SE_t}b(\SE_t(\cdot))\right)\cdot\EE[\nu_t^\tau\otimes \nu_t^\tau]+O(t^2)
        \end{split}
    \end{equation}
    and 
    \begin{equation}\label{E14}
        \begin{split}
       & \f{d}{dt} \EE[\nu_t^\tau\otimes \nu_t^\tau]\\&=\s^2\Id+
       \left\{
       \left(2T_{\SE_t}b(\SE_t (\cdot))+4\n_\cdot b(\SE_t)\right)\odot \Id
       +\f12 R(\cdot^\sharp,\cdot)\cdot
       \right\}
       \EE[\nu_t^\tau\otimes \nu_t^\tau]+ O(t^2).
        \end{split}
    \end{equation}

\end{theorem}

\begin{proof}
Note that, on the algebra $\SG$ associated to $SO(3)$, any orthonormal basis $(G_1,G_2,G_3)$ satisfies up to a change for the index numbering:
\begin{equation}\label{E2.14}
[G_1,G_2]=\frac{G_3}{\sqrt{2}};\ [G_2,G_3]=\frac{G_1}{\sqrt{2}};\  [G_3,G_1]=\frac{G_2}{\sqrt{2}}.
\end{equation}
The first part of Theorem \ref{T3} is a direct application of Theorem $2$ considering that, for any $\nu\in\SG$:
\begin{equation}\label{E2.15}
    \sum\limits_{i=1}^n\left[\left[[G_i,\nu],\nu\right],G_i\right]
    =0
    \text{ and }
    \sum\limits_{i=1}^n[G_i,\nu]\otimes[G_i,\nu]=\frac{\|\nu\|^2}{2}\Id_{\nu ^{\perp}}.
\end{equation}
Indeed, as these values do not depend on the choice of the orthonormal basis, we can choose $(\tilde G_1,\tilde G_2,\tilde G_3)$ satisfying \eqref{E2.14} such that $\tilde G_1=\frac{\nu}{\|\nu\|}$. With this choice for a basis the above results are immediate.

To prove the second part of Theorem \ref{T3} we notice that, as the chosen metric is bi-invariant, $\nabla_X \nu=\frac{1}{2}[X,\nu]$ and $R(X,\nu)\nu=-\frac{1}{4}\left[[X,\nu],\nu\right]$ for any $X,\nu\in\SG$. In particular, decomposing $X$ in the basis $(\tilde G_1,\tilde G_2,\tilde G_3)$, we get $\left[[X,\nu],\nu\right]=-\frac{1}{2}\|\nu\|^2P_{\nu^{\perp}}(X)$ and $R(X,\nu)\nu=\frac{1}{8}\|\nu\|^2P_{\nu^{\perp}}(X)$.
\qed
\end{proof}

\section{Proof of Theorem~\ref{T2}}
\label{S3}
 From Theorem~\ref{T1}, assumptions on $X_t$ and Equation~\eqref{E10} for $\nu_t^\tau$ we can deduce that for all $k\ge 1$,
 \begin{equation}
     \label{normnu}
     \|\nu^\tau\|_{t,k}^k:=\E\left[\sup_{s\le t}|\nu_s^\tau|^k \right]\le C_k t^{k/2}
 \end{equation}
 for some $C_k>0$.
In the sequel, 
\begin{equation}\label{betaalpha}
R_{t,k}=\int_0^t\beta_k(s)ds+\int_0^t\sum_{i=1}^m\alpha_{k,i}(s) dW_s^i,
\end{equation}
with $\beta(s)$ and $\alpha_{k,i}(s)$ taking their values in $\SG$, will denote error terms satisfying 
\begin{equation}
    \label{Rtk}
    \EE\left[\sup_{s\le t}\left(\sum_{i=1}^m|\alpha_{k,i}(s)|^k+|\beta_k(s)|^k\right)\right]\le C \|\nu^\tau\|_{t,k}^k 
\end{equation} 
for some $C>0$,
and will change from a line to another. 

Denote $S_t$ the $\SG$-valued semimartingale defined by $S_0=0$ and  $\circ dS_t=\exp(-\nut_t)\circ d\left(\exp(\nut_t)\right)$. 
By derivating \eqref{E1}, we get:
\begin{align}\label{E2.1}
\circ d\Xt_t&=d\SE_t\exp(\nut_t)+\SE_t\circ d\left(\exp(\nut_t)\right)
\notag\\
&=\Xt_t\left(\exp\left(\ad(-\nut_t)\right)(h_t)dt+\circ dS_t\right).
\end{align}
Using \eqref{E3}, we obtain the following relation on $\SG$:
\begin{equation*}
    b(\Xt_t)dt+\circ d\Wt_t=\exp\left(\ad(-\nut_t)\right)(h_t)dt+\circ d S_t
\end{equation*} and thus, passing from Stratonovich to Itô:
\begin{equation}\label{E2.2}
    b(\Xt_t)dt+ dW_t^{\tau}=\exp\left(\ad(-\nut_t)\right)(h_t)dt+ d S_t.
\end{equation}
 We first look at the left-hand side term. As $\Xt_t=\SE_t\exp(\nut_t)$, we get:
\begin{align}\label{E2.3}
    b(\Xt_t)
    &=b(\SE_t)+T_{\SE_t}b\left(\SE_t\nut_t\right)+\frac{1}{2}\Hess_{\SE_t}b\left(\SE_t\nut_t\otimes\SE_t\nut_t\right)+O(|\nu_t^\tau|^3).
\end{align}
We now turn to the right-hand side term. We first have:
\begin{equation}\label{E2.4}
\exp\left(\ad(-\nut_t)\right)(h_t)=h_t+[h_t,\nut_t]+\frac{1}{2}\left[[h_t,\nut_t],\nut_t\right]++O(|\nu_t^\tau|^3).
\end{equation}
For $u\in\SG$, the differential of the exponential map at point $u$ is $d_u\exp=\exp(u)\sum\limits_{k\geq 0}\frac{\ad(-u)^{(k)}}{(k+1)!}$ where $\ad(-u)^{(k)}$ denote the $k$-th iteration of $\ad(-u)$. Then:
\begin{equation}\label{E2.5}
    \circ dS_t=\exp(-\nut_t)\circ d \left(\exp(\nut_t)\right)=\sum\limits_{k\geq 0}\frac{\ad(-\nut_t)^{(k)}}{(k+1)!}(\circ d\nut_t).
\end{equation}
Passing from Stratonovich to Itô and using the Jacobi identity to show that 
$\left[\left[[X,Y],Y\right],X\right]=-\left[\left[X,[X,Y]\right],Y\right]=\left[\left[[X,Y],X\right],Y\right]$:

\begin{align}\label{E2.6}
    dS_t
    &=d\nut_t+\frac{1}{2}[d\nut_t,\nut_t]+\frac{1}{12}\left[[d\nut_t,\nut_t],\nut_t\right]+\frac{1}{6}\left[[d\nut_t,\nut_t],d\nut_t]\right]\notag\\
    &+\frac{1}{24}
    \left[\left[[d\nut_t,\nut_t],d\nut_t\right],\nut_t\right]
    +dR_{t,3}.
\end{align}
Write the decomposition $\nut_t=M_t+D_t$ where $M_t$ is the martingale part of $\nut_t$ and $D_t$ its finite variation drift part. Putting together \eqref{E2.3}, \eqref{E2.4} and \eqref{E2.6}, \eqref{E2.2}, we have:
 \begin{align}
     dM_t&=d\Wt_t-\frac{1}{2}[dM_t,\nut_t]-\frac{1}{12}\left[[dM_t,\nut_t],\nut_t\right]+dR_{t,3};\label{E2.7}\\
     \begin{split}
         dD_t&=-h_tdt+b(\SE_t)dt+T_{\SE_t}b\left(\SE_t\nut_t\right)dt-[h_t,\nut_t]dt-\frac{1}{6}\left[[dM_t,\nut_t],dM_t]\right]\\
    &-\frac{1}{2}[dD_t,\nut_t]-\frac{1}{12}\left[[dD_t,\nut_t],\nut_t\right]
   -\frac{1}{2}\left[[h_t,\nut_t],\nut_t\right]dt\\
    &+\frac{1}{2}\Hess_{\SE_t}b\left(\SE_t\nut_t\otimes\SE_t\nut_t\right)dt-\frac{1}{24}\left[\left[[dM_t,\nut_t],dM_t\right],\nut_t\right]+dR_{t,3}.
     \end{split}\label{E2.8}
 \end{align}

Applying $\ad(-\nut_t)$ twice and once to \eqref{E2.7}, we get:
\begin{equation}
    \left[[dM_t,\nut_t],\nut_t\right]=\left[[d\Wt_t,\nut_t],\nut_t\right]+dR_{t,3}
    \end{equation}
    \begin{equation}
  \begin{split}  [dM_t,\nut_t]
  &=[d\Wt_t,\nut_t]-\frac{1}{2}\left[[d\Wt_t,\nut_t],\nut_t\right]+dR_{t,3}.
    \end{split}
\end{equation}
This provides a nice estimate for $dM_t$:
\begin{equation}
\begin{split}
    dM_t
    &=d\Wt_t-\frac{1}{2}[d\Wt_t,\nut_t]+\frac{1}{6}\left[[d\Wt_t,\nut_t],\nut_t\right]+dR_{t,3}.
\end{split}
\end{equation}
In particular we obtain:
\begin{equation}\label{E2.9}
   \begin{split}d\nut_t\otimes &d\nut_t=dM_t\otimes dM_t
   =\sum\limits_{i=1}^m\left(\sigt_i\otimes \sigt_i-
   \sigt_i\odot[\sigt_i,\nut_t]\right.\\
   &\left.+\left(\frac{1}{3}\sigt_i\odot\left[[\sigt_i,\cdot],\cdot\right]+\frac{1}{4}[\sigt_i,\cdot]\otimes[\sigt_i,\cdot]\right)\left(\nut_t\otimes\nut_t\right)\right)dt+dR_{t,3}.
   \end{split}
\end{equation}
We now look for the estimate of $dD_t$. From \eqref{E2.9}, we get:
\begin{align*}
    &\left[[dM_t,\nut_t],dM_t\right]=\sum\limits_{i=1}^m\left(\left[[\sigt_i,\nut_t],\sigt_i\right]-\frac{1}{2}\left[\left[[\sigt_i,\cdot],\cdot\right],\sigt_i\right]\left(\nut_t\otimes\nut_t\right)\right)dt+dR_{t,3}\\
    &\left[\left[[dM_t,\nut_t],dM_t\right],\nut_t\right]=\sum\limits_{i=1}^m\left[\left[[\sigt_i,\cdot],\sigt_i\right],\cdot\right]\left(\nut_t\otimes\nut_t\right)dt+dR_{t,3}
\end{align*}
Applying $\ad(-\nut_t)$ twice and once to \eqref{E2.8}, we get:
\begin{equation}
    \left[[dD_t,\nut_t],\nut_t\right]=\left[[-h_t+b(\SE_t),\nut_t],\nut_t\right]dt+dR_{t,3}
    \end{equation}
    and
    \begin{equation}
  \begin{split}  [dD_t,\nut_t]
  &=[-h_t+b(\SE_t),\nut_t]dt+\left[T_{\SE_t}b(\SE_t\cdot)\right.\\
  &\left.-\frac{1}{2}[h_t+b(\SE_t),\cdot]-\frac{1}{6}\sum\limits_{i=1}^m\left[[\sigt_i,\cdot],\sigt_i\right],\cdot\right]\left(\nut_t\otimes\nut_t\right)dt+dR_{t,3}.
    \end{split}
\end{equation}
Finally:
{\begin{align}\label{E2.10}
 \frac{d}{dt}&D_t
  =-h_t+b(\SE_t)+\left(T_{\SE_t}b\left(\SE_t\cdot\right)-\frac{1}{2}[h_t+b(\SE_t),\cdot]\right)(\nut_t)\notag\\
    &-\frac{1}{6}\sum\limits_{i=1}^m\left[[\sigma_i,\cdot],\sigma_i\right]\left(\mathbb 1_{\{t\le\tau\}}\nut_t\right)
    +\left(\left[-\frac{1}{2}T_{\SE_t}b(\SE_t\cdot)+\frac{1}{6}[-h_t+b(\SE_t),\cdot],\cdot\right]\right.\\
     &\left.+\frac{1}{2}\Hess_{\SE_t}b\left(\SE_t\cdot\otimes\SE_t\cdot\right)+\frac{1}{8}\sum\limits_{i=1}^m\left[\left[[\sigma_i,\cdot],\cdot\right],\sigma_i\right]
     \right)\left(\mathbb 1_{\{t\le\tau\}}\nut_t\otimes\nut_t\right)
     +\beta_3(t).\notag\end{align}}
with $\beta_3$ defined in~\eqref{betaalpha}.
From Definition \ref{D1}, we have $\esp[\nut_t]=0$. In particular, $\esp[D_t]=0$. {Moreover, similarly to Lemma~2.3 in~\cite{ATW:09}, we can prove that $\PP(\tau\le t)\le Ce^{-C'/t}$ for some $C,C'>0$. Then there exists $C''>0$ such that 
\[
\esp[\mathbb 1_{\{t\le \tau \}}\nut_t\otimes \nut_t]=\esp[\nut_t\otimes \nut_t]+O\left(e^{-\frac {C''}{t}}\right)
\]
and $\esp[\mathbb 1_{\{t\le \tau\}}\nut_t]=O\left(e^{-\frac {C''}{t}}\right)$.} Looking at the mean of \eqref{E2.10}   we obtain:
\begin{align}\label{E2.11}
        h_t&=b(\SE_t)+\left(\left[-\frac{1}{2}T_{\SE_t}b(\SE_t\cdot)+\frac{1}{6}[-h_t+b(\SE_t),\cdot],\cdot\right]+\frac{1}{2}\Hess_{\SE_t}b\left(\SE_t\cdot\otimes\SE_t\cdot\right)\right.\notag\\
     &\left.+\frac{1}{8}\sum\limits_{i=1}^m\left[\left[[\s_i,\cdot],\cdot\right],\s_i\right]
     \right)\esp\left[\nut_t\otimes\nut_t\right]+O(t^{3/2}).
\end{align}
In particular, $h_t=b(\SE_t)+O(t)$ and $\left[[-h_t+b(\SE_t),\cdot],\cdot\right]\esp[\nut_t\otimes\nut_t]=O(t^{3/2})$. Now $h_t$ being smooth, $O(t^{3/2})$ implies $O(t^2)$. This proves Relation \eqref{E6}.

We are left to obtain the estimate of the error term.
Using the Itô relation, we have: \begin{equation}\label{E2.12}
d\esp[\nut_t\otimes\nut_t]=\esp[2dD_t\odot\nut_t+d\nut_t\otimes d\nut_t]
\end{equation}
We have:
\begin{align}\label{E2.13}
    dD_t\odot\nut_t&=\left(-h_t+b(\SE_t)\right)\odot\nut_tdt+\Big(T_{\SE_t}b\left(\SE_t\cdot\right)\odot \Id-\frac{1}{6}\sum\limits_{i=1}^m\left[[\sigt_i,\cdot],\sigt_i\right]\odot \Id\notag\\
    &-\f12[h_t+b(\SE_t),\cdot]\odot \Id\Big)(\nut_t\otimes\nut_t)dt+dR_{t,3}.
\end{align}

To obtain \eqref{E7}, we then just have to look at the means of  \eqref{E2.9} and \eqref{E2.13} and use again that $h_t=b(\SE_t)+O(t)$, $\esp[\nut_t]=0$ and $t\mapsto\esp[\nut_t\otimes \nut_t]$ is smooth (with the same arguments as for the smoothness of $t\mapsto\SE_t$).


\section{Simulations}\label{Ssim}
We performed Matlab simulations to validate the proposed approach in the case of a rotation matrix $X_t$ that evolves on the Lie group $SO(3)$. For that purpose, we simulated $N_{MC}=500$ realisations of \eqref{E3} by taking $b(X)=X^{-1}AX$, with $A$ a given antisymetric matrix, and a standard deviation $\sigma=0.1$ for the Brownian motions. A total time interval of $0.1$ s was considered, decomposed in $N=100$ steps for the discretization. Finally, we computed the Fr\'echet mean of $X_t$ either by leveraging the proposed set of equations \eqref{E13} and \eqref{E14}, or, alternatively, by using a Gauss-Newton on Lie group to minimize an empirical approximation of \eqref{E2} based on the $N_{MC}$ Monte Carlo runs. To visually compare both obtained rotation matrices at the last time step, we represent in Figure \ref{fig:1} their effect on the unitary vectors $(\mathbf{e}_1, \mathbf{e}_2, \mathbf{e}_3)$ of the canonical basis of $\mathbb{R}^3$ (plotted as black arrows). The differently colored clouds of points are obtained by applying the $N_{MC}$ simulated matrices to $\mathbf{e}_1, \mathbf{e}_2$ and $\mathbf{e}_3$, respectively. The blue and red lines, that are nearly superimposed, show the three transformed basis vectors using the theoretical and Monte Carlo Fr\' echet means, respectively, and a very good coincidence between them can be observed.
\begin{figure}[!htb]
    \centering
    \includegraphics[width=.7\linewidth]{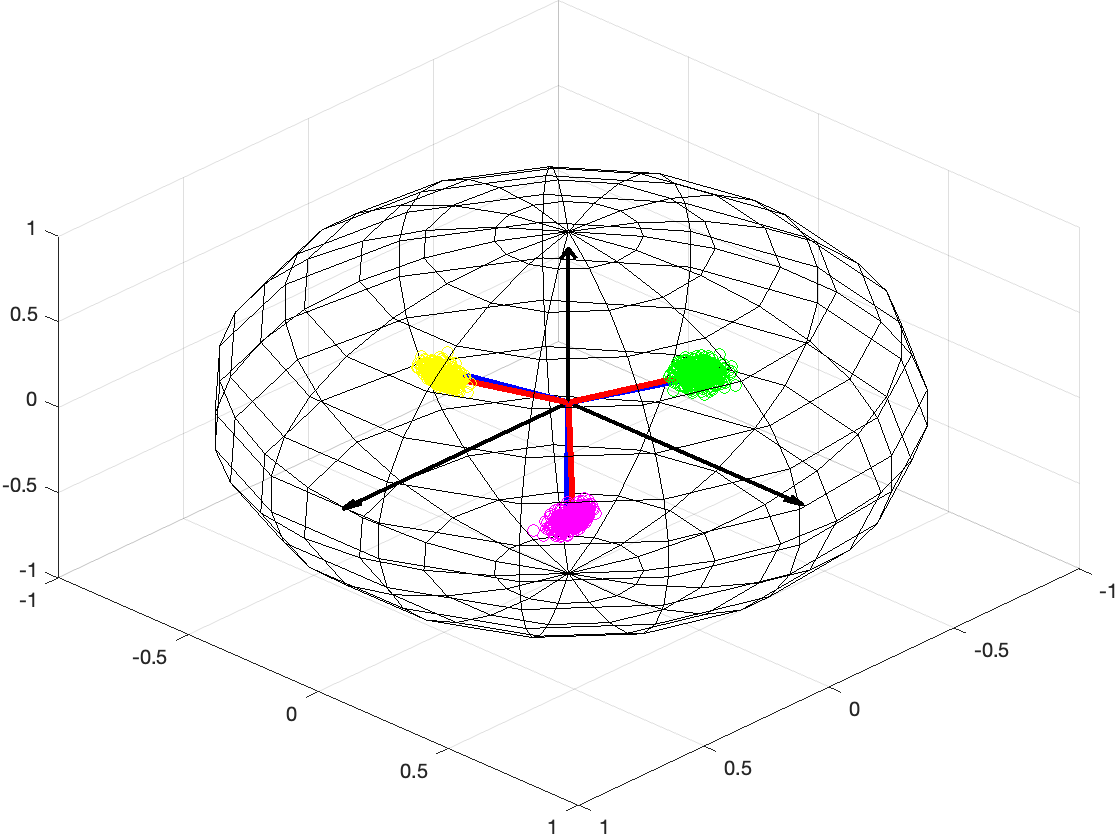}
    \caption{Comparison between the empirical and the proposed Fr\'echet mean.}
    \label{fig:1}
\end{figure}

\section{Conclusion}\label{Sfinal}
In this paper, we focused on the prediction step of a continuous-time filter on Lie group. An extension of our results will be at first to let the diffusion process $X_t$ start from a non Dirac law. The next step will be to update the Fréchet mean of the posterior law by taking into account the current measurement. Finally, let us stress the fact that our method requires a minimal set of assumptions, which allows to extend it to more general state spaces, as symmetric spaces, homogeneous spaces and Riemanian manifolds.

\bibliographystyle{splncs04}
\bibliography{Bibliographie}
\end{document}